\documentclass[a4paper]{amsart}
\usepackage[utf8]{inputenc}
\usepackage[english]{babel}
\usepackage [autostyle, english = american]{csquotes}
\MakeOuterQuote{"}

\usepackage{amssymb,mathrsfs,amsmath,graphicx,amsthm,microtype,tikz-cd,bbm,wasysym,enumitem,comment,cases,hyperref,subcaption,caption,mathrsfs,tkz-graph,float}

\usepackage{csquotes}
\usepackage{yhmath}
\usepackage{cancel}
\usepackage{color}
\usepackage{siunitx}
\usepackage{array}
\usepackage{multirow}
\usepackage{amssymb}
\usepackage{gensymb}
\usepackage{tabularx}
\usepackage{extarrows}
\usepackage{booktabs}
\usetikzlibrary{fadings}
\usetikzlibrary{patterns}
\usetikzlibrary{shadows.blur}
\usetikzlibrary{shapes}
\usepackage[
backend=biber,
style=alphabetic,
sorting=nyt
]{biblatex}
 
\usetikzlibrary{babel}	
\tikzset{%
  symbol/.style={
    draw=none,
    every to/.append style={
      edge node={node [sloped, allow upside down, auto=false]{$#1$}}
    },
  },
}

\newtheorem{thm}{Theorem}[section]
\newtheorem{cor}[thm]{Corollary}
\newtheorem{lem}[thm]{Lemma}
\newtheorem*{lem*}{Lemma}
\newtheorem{defi}[thm]{Definition}
\newtheorem{obs}{Observation}
\newtheorem{prop}[thm]{Proposition}

\newtheorem{ej}[thm]{Example}
\newtheorem*{thm*}{Teorema}
\newtheorem*{remark}{Remark}

\date{}
\title{When the Weak Separation Condition implies the Generalize Finite Type in $\R^d$}
\author{Kevin G. Hare}
\address{Department of Pure Mathematics \\
          University of Waterloo \\
          Waterloo, Canada}
\email{kghare@uwaterloo.ca}
\thanks{Research of K.G. Hare was supported, in part, by NSERC Grant 2025-03965}
\author{Joaquin G. Prandi}
\address{Department of Pure Mathematics \\
          University of Waterloo \\
          Waterloo, Canada}
\email{jgprandi@uwaterloo.ca, prandijoaquin@gmail.com}
\thanks{Research of J. G. Prandi was supported, in part, by NSERC Grant 2025-03965}

\def\E{\mathcal{E}}
\def\F{\mathcal{F}}

\def\cG{\mathcal{G}}

\def\R{\mathbb{R}}
\def\cS{\mathcal{S}}
\def\sS{\mathscr{S}}

\def\Z{\mathbb{Z}}
\def\K{[-1/2,1/2]^d}
\def\V{(-1/2,1/2)^d}
\def\dim{\operatorname{dim}}
\def\supp{\operatorname{supp}}
\def\ldim{\underline{\operatorname{dim}}_{loc}}
\def\udim{\overline{\operatorname{dim}}_{loc}}
\def\Assumptions{Let $\{S_i\}_{i=0}^n$ be an IFS such that $S_i=r_iA_ix+d_i$, with attractor $\K$, $0<r_i<1$. Further, assume that $A_i$ fixes $\K$}
\def\inte{\operatorname{int}}
\newcounter{relctr} %% <- counter for relations
\everydisplay\expandafter{\the\everydisplay\setcounter{relctr}{0}} %% <- reset every eq
\AtBeginDocument{} %% <- store original definition
\graphicspath{ {Pictures/}}

\begin{document}
\begin{abstract}

Let $\cS$ be an iterated function system in $\R^d$, with full support and some restrictions on the allowable rotations.  We show that $\cS$ satisfies the weak separation condition if and only if it satisfies the generalized finite-type condition. With this in mind, we extend the notion of net intervals from $\R$ to $\R^d$. We also use net intervals to calculate the local dimension of a self-similar measure with the finite-type condition and full support. 
\end{abstract}
\maketitle

%$\cS$ cS is IFS

%$\Lambda_\alpha$ is words of depth $\alpha$

%$\sS_\alpha=\{S_\sigma\}_{\sigma\in\Lambda_\alpha}$ sS

%E= E cS(V)=$\E= \E_\cS(V)$

\section{Introduction}
Consider an Iterated Function System (IFS) comprised of similarities $\{S_i\}_{i=1}^k$ on $\R^d$. It is well known (see \cite{falconer}) that there is a unique non-empty compact set $K$, called the \emph{attractor} (also called the \emph{self-similar set}) of the IFS, such that \[K=\bigcup_i S_i(K).\] The dimensional properties of such sets are of great interest. If the IFS satisfied the Open Set Condition (OSC), then a great deal is understood, and several techniques are known. The Sierpinski triangle is one such example. Many interesting examples do not have this property. In \cite{WSPfirstappearance} the weak separation condition (WSC) was introduced. This condition allows for limited types of overlap. In \cite{weakseparationequivalence} various equivalent definitions of the WSC were presented. The geometrical properties of IFS with the WSC are still difficult to prove. When combinatorial properties are introduced to the overlaps, much more is known. This is how the finite-type condition (FTC) and the generalized finite-type condition (GFTC) were introduced in \cite{FTCfirstappearance} and \cite{GFTCfirstappearance}, respectively. The IFSs that satisfy the generalized finite-type condition have proven to be more tractable. 

We know that OSC implies FTC; the FTC implies GFTC which, in turn, implies the WSC. In \cite{LyapunovFeng}, it was shown that if the contraction ratios are equal, positive, and the attractor is the interval $[0,1]$, then the WSC is equivalent to the GFTC. Later in \cite{hare/hare/rutar}, this result was extended to any kind of contraction and support $[0,1]$. We also note that the techniques in \cite{hare/hare/rutar} are related to the calculation of the local dimension of self-similar measures that were used in \cite{HareHareMatthews}. In particular, in \cite{HareHareMatthews} an algorithm is developed to calculate the local dimension of a self-similar measure.

In \cite{hare/hare/rutar} the use of net intervals is key. The net intervals are determined by a neighbor set. Connecting the notion between GFTC and WSC, the authors of \cite{hare/hare/rutar} propose a new notion called finite neighbor condition (FNC). They proceed to show that the FNC is equivalent to the GFTC, and finally they prove that the WSC implies the FNC. A cornerstone of the work done is the concept of net intervals and neighbor sets; we note that the definition of the net interval in $\R$ does not extend to $\R^d$. In Definition \ref{netinterval} we propose a new definition of the net interval and the neighbor set, which works in $\R^d$. It is important to note that this definition is not the only one possible. In $\R$, there are examples with IFS for which the classical definition and the one we propose give a different net interval. Although not exactly the same definition as the classical one, it conserves the same kind of good properties as the classical one.

We organize this paper as follows. In Section \ref{Sdefinitios} we give the definitions that we will use in the rest of this paper. This includes the definition that we propose for net intervals and neighbor sets. In Section \ref{SCharact} we give different characterization of the different overlapping conditions. In Section \ref{SEquiv} we show the main result of this paper, when the IFS has full support and the rotations fix the hypercube, then the WSC implies the GFTC\textsubscript{co}. 

 Section \ref{SMeasures} does not follow directly from Section \ref{SEquiv}. Instead, it is a practical application on how to compute the local dimension of self-similar measures. To be more precise, we will extend the algorithm to calculate the local dimension of self-similar measures presented in \cite{HareHareMatthews} from $\R$ to $\R^d$. We give the definitions of self-similar measures and the local dimension of a measure at a point. We then define transition matrices and associated a weighted directed graph with the system. We show that the associated graph has a unique essential class. This allows us to conclude that the graph we present here has the same properties as the one in \cite{HareHareMatthews}. We note that the graphs in \cite{HareHareMatthews} are for IFS in $\R$, we work in $\R^d$. The graphs obtained here for $\R$ are equivalent to the ones in \cite{HareHareMatthews}. We can also see that the properties we obtain from the graphs for $\R^d$, are similar to those for the one in $\R$. In addition, with some care, the same techniques developed in \cite{HareHareMatthews} will follow for IFS in $\R^d$.
%\part{WSP implies the GFTC}

\section{Definitions}\label{Sdefinitios}

In this section, we present the definitions we will work with in this paper. We will start with the definitions of an IFS. We then define the open set condition (OSC), weak separation condition (WSC), and (generalized) finite-type condition ((G)FTC). At this point, we present our definition of net intervals and the neighbor set. We close this section with the definition of the finite neighbor condition that we will use in this paper.

We recall that an IFS is a set $\cS=\{S_i\}_{i=0}^n$ of similarities in $\R^d$. That is, $\|S_i(x)-S_i(y)\|=r_i\|x-y\|$ with $r_i\in (0,1)$ for $i\in\{0,\dots,n\}$. From this we can deduce $S_i(x)=r_i A_i x+d_i$ where $A_i$ are orthogonal matrices in $\R^d$ and $d_i\in \R^d$.

We also assume that an IFS in this paper has the attractor $K\subset[-1/2,1/2]^d$ and that the closure of the convex hull of $K$, $\operatorname{cl(hull(}K))$, is the hypercube, unless otherwise stated.

 We recall that an IFS for which all contraction ratios are equal is called \emph{equicontractive}.
 
\begin{defi}\label{OSC}
    An IFS $\cS=\{S_i\}_{i=1}^k$ is said to satisfy the \emph{open set condition} (OSC) if there is a non-empty bounded open set $V$ such that $S_i(V)\subset V$ for all $i$, and $S_i(V)\cap S_j(V)=\emptyset$ for all $i\neq j$. If $V=\inte(\operatorname{hull}(K))$, then we say that $\cS$ satisfies the \emph{convex open set condition} (OSC\textsubscript{co}).
\end{defi}

We say that a set $V$ is \emph{invariant} under $\cS=\{S_i\}_{i=1}^k$ if $S_i(V)\subset V$ for all $i$. Hence, $\cS$ satisfies the OSC if and only if there is a non-empty bounded invariant open set $V$ such that $S_i(V)\cap S_j(V)=\emptyset$ if $i\neq j$.

The OSC tells us that the different parts of the attractor do not interact meaningfully with each other as one iterates the IFS. The weak separation condition and the finite-type condition allow for overlaps. To give these definitions, we need additional notation.

Let $\Sigma=\{1,\dots,k\}$ and $\Sigma^*$ denote the set of all finite words in $\Sigma$. We use $\epsilon$ for the \emph{empty word}. We note that $\epsilon$ is the word of length $0$. Given $\sigma=(\sigma_1,\dots,\sigma_j)\in \Sigma^*$, we put
\[\sigma^-=(\sigma_1,\dots,\sigma_{j-1}),\ S_\sigma=S_{\sigma_1}\circ\dots\circ S_{\sigma_j}\ \text{ and }\ r_\sigma=\prod_{i=1}^j r_{\sigma_i}.\]
We also note that $S_\epsilon\text{ is the identity}.$
Given $\alpha>0$, put $\Lambda_\alpha=\{\sigma \in \Sigma^*:r_\sigma<\alpha\leq r_{\sigma^-}\}$ and set $\sS_\alpha=\{S_\sigma\}_{\sigma\in\Lambda_\alpha}$.
We also differentiate the smallest contraction ratio as $r_{\min}=\min_i\{r_i\}.$

The next is one of multiple equivalent definitions for the weak separation condition. For a more complete list of equivalences, see \cite{weakseparationequivalence}.
\begin{defi}\label{def 2.2}
    An IFS is said to satisfy the \emph{weak separation condition} (WSC) if there is some $x_0\in \R^d$ and integer $N$ such that for $\alpha$ and a finite word $\tau$, any close ball with radius $\alpha$ does not contain more than $N$ distinct points of the form $S_\sigma(S_\tau(x_0))$ for $\sigma\in \Lambda_\alpha$.
\end{defi}

\begin{obs}
    Equivalently, an IFS satisfies the WSC if for all $x$ there is some $N$ such that for $\alpha$ and a finite word $\tau$, any close ball of radius $\alpha$ contains at most $N$ distinct points of the form $S_\sigma(S_\tau(x))$ for $\sigma\in \Lambda_\alpha$. 

Note that we can change the balls in the definition of WSC to hypercubes or to hypercubes in a grid. To see this we note any ball of radius $\alpha$ is contained in a hypercube of diameter $\sqrt{d}2\alpha$, and any cube of diameter $\alpha$ is contained in a sphere of radius $\alpha$. 

We say a hypercube is on a grid if there is an $r\in \R^+$ such that the vertices of the hypercube are contained in $r\Z^d$. Any hypercube of diameter $r$ is contained in the union of at most $2^d$ hypercubes on a grid of diameter $r$. Hence, we can also change the balls in the definition for hypercubes on a grid in the definition of WSC.
\end{obs}

Although it is not clear from the Definition \ref{def 2.2}, the idea behind the WSC is that there is, at different levels, a lower bound on how small the overlaps can be.
\begin{defi}
    For any $V\subset \R^d$ let 
    \begin{equation}\label{eq2.2}
        \E_\cS(V)=\bigcup_{\alpha>0}\left\{S_\sigma^{-1}\circ S_\tau:\sigma,\tau\in\Lambda_\alpha,S_\sigma(V)\cap S_\tau(V)\neq\emptyset\right\}.
    \end{equation}
\end{defi}
\begin{obs}
    If an IFS $\cS=\{S_i\}$ satisfies the OSC with the invariant open set $V$, then $\E_\cS(V)$ is the identity map.
\end{obs}
\begin{defi}\label{GFTC}
    An IFS $\cS=\{S_i\}$ is said to satisfy the \emph{generalized finite type condition} (GFTC), if $\E_\cS(V)$ is finite for some non-empty bounded invariant open set $V$.

    An IFS $\cS=\{S_i\}$ is said to satisfy \emph{finite type condition} (FTC), if $\E_\cS(V)$ is finite for some non-empty bounded invariant open set $V$ and the contraction ratios are logarithmically commensurate (i.e. $\log(r_i)/\log(r_j)\in \mathbb{Q}$).

    When $\inte(\operatorname{hull}(K))\neq\emptyset$, we say that $\cS$ satisfies the \emph{convex generalized finite type condition} (GFTC\textsubscript{co}) if it satisfies GFTC with the invariant open set being the interior of the convex hull of the attractor $K$. If $\inte(\operatorname{hull}(K))=\emptyset$, then $K$ is contained in a subspace of $\R^d$.
\end{defi}
 This is not the original definition of GFTC, but was shown to be equivalent in \cite{equivfinitetype}. The idea behind these definitions is that, once we normalize, the overlaps are of finitely many types. 

 Let $\cS=\{S_i\}_{i=0}^n$ be an equicontractive IFS such that $S_i(x)=r A_ix+d_i$. We define 
\[N_V(S_\sigma)=\{S_\sigma^{-1}\circ S_\tau: |\sigma|=|\tau|\text{ and }S_\sigma(V)\cap S_\tau(V)\}.\]
We note that $\E= \E_\cS(V)$ is finite if only there are finitely many $N_V(S_\sigma)$.

These notions will be used in, for example, Lemma \ref{gencharactgftc}.

\begin{remark}
    It is known that
 \begin{center}
\tikzset{every picture/.style={line width=0.75pt}} %set default line width to 0.75pt        
\begin{tikzpicture}[x=0.75pt,y=0.75pt,yscale=-1,xscale=1]
%uncomment if require: \path (0,336); %set diagram left start at 0, and has height of 336

% Text Node
\draw (121,122.4) node [anchor=north west][inner sep=0.75pt]    {$OSC$};
% Text Node
\draw (121,172.4) node [anchor=north west][inner sep=0.75pt]    {$OSC_{co}$};
% Text Node
\draw (221,122.4) node [anchor=north west][inner sep=0.75pt]    {$GFTC$};
% Text Node
\draw (222,172.4) node [anchor=north west][inner sep=0.75pt]    {$GFTC_{co}$};
% Text Node
\draw (330,122.4) node [anchor=north west][inner sep=0.75pt]    {$WSC$};
% Text Node
\draw (181,122.4) node [anchor=north west][inner sep=0.75pt]    {$\subseteq $};
% Text Node
\draw (281,122.4) node [anchor=north west][inner sep=0.75pt]    {$\subseteq $};
% Text Node
\draw (132.4,164) node [anchor=north west][inner sep=0.75pt]  [rotate=-270]  {$\subseteq $};
% Text Node
\draw (232.4,164) node [anchor=north west][inner sep=0.75pt]  [rotate=-270]  {$\subseteq $};
% Text Node
\draw (181,172.4) node [anchor=north west][inner sep=0.75pt]    {$\subseteq $};

\end{tikzpicture}

\end{center}
  \end{remark}

We will now introduce the interrelated concepts of net intervals and neighbor sets. 

The underlying concept of net intervals is of great use in the literature, especially when talking about IFS with overlaps. The idea behind a net interval is to divide the attractor $K$ into smaller pieces. We can think of the attractor in levels, and these levels are determined by the iterations of the IFS. At a given level, we want to divide $K$ by the overlaps of $S_\sigma(K)$. In other words, we want the pieces to be determined by how many covers they have. The net intervals are determined by what $S_\sigma(K)$ covers it. That is, all points in a net interval have the same number of $S_\sigma(K)$ as a covering. 

 We propose the next definition of net intervals in $\R^d$.  We define net intervals in the case that the attractor $K$ of the IFS $\cS$ has as convex hull $\K$.
\begin{defi}\label{netinterval}
    Let $\cS=\{S_i\}$ be an IFS such that the self-similar set is $K\subset\K$. Let $\delta>0$ and $V=\V$. For $x\in \K\setminus \cup_{\sigma\in \Lambda_\delta}S_\sigma(\partial\K)$ define the \emph{net interval} of $x$ at level $\delta$ as
\begin{equation}\label{eqninterval}
    \Delta_\delta(x)=\left(\bigcap_{\substack{T\in \sS_\delta \\ x\in T(V)}}T(\overline{V})\right)\setminus \left(\bigcup_{\substack{T\in \sS_\delta\\ x\notin T(V)}}T(V)\right) \text{ and }\inte(\Delta_\delta(x))\cap K\neq\emptyset.
\end{equation}
Define $\F_\delta$ as the collection of all net intervals at level $\delta$. 

If $x\in \cup_{\sigma\in \Lambda_\delta}S_\sigma(\partial\K)$ for some $\delta$, then $x\in\cup_{\sigma\in \Lambda_{\delta'}}S_\sigma(\partial\K)$ for all $\delta'\leq\delta$ or $x\notin K$. In addition, any $x\in\cup_{\sigma\in \Lambda_\delta}S_\sigma(\partial\K)$ will be contained in multiple net intervals.
\end{defi}
\begin{ej}\label{exnetinterval}
     Consider the IFS given by
        \begin{align*}
        S_1(x,y)&=1/2(x,y)+(-1/4,1/4),   &S_3(x,y)=1/2(x,y)+(1/4,1/4),\\
        S_2(x,y)&=1/3(x,y), &S_4(x,y)=1/2(x,y)-(1/4,1/4),\\
        S_5(x,y)&=1/2(x,y)+(1/4,-1/4).        
    \end{align*}
    Then the attractor of the IFS is $[-1/2,1/2]\times[-1/2,1/2]$. For $1>\alpha>1/2$ we have $\sS_\alpha=\{S_1,S_2,S_3,S_4,S_5\}$. Consider the point $\textbf{x}_1=(-1/3,1/4)$. Then $\textbf{x}_1\in S_1(V)$ and not in $S_i(V)$ for $i=2,3,4,5$. Further $\Delta_\alpha(\textbf{x}_1)=S_1([-1/2,1/2]\times[-1/2,1/2])\setminus\cup_{i=2}^5S_i(V)$. In Figure \ref{firstexample}, we can see on the left the image of $[-1/2,1/2]\times[-1/2,1/2]$ by $\{S_1,S_2,S_3,S_4,S_5\}$ and on the right we see $\Delta(\textbf{x}_1)$.

    \begin{figure}[h]
    \centering

% Pattern Info
 
\tikzset{
pattern size/.store in=\mcSize, 
pattern size = 5pt,
pattern thickness/.store in=\mcThickness, 
pattern thickness = 0.3pt,
pattern radius/.store in=\mcRadius, 
pattern radius = 1pt}
\makeatletter
\pgfutil@ifundefined{pgf@pattern@name@_abi33ot3k}{
\pgfdeclarepatternformonly[\mcThickness,\mcSize]{_abi33ot3k}
{\pgfqpoint{-\mcThickness}{-\mcThickness}}
{\pgfpoint{\mcSize}{\mcSize}}
{\pgfpoint{\mcSize}{\mcSize}}
{
\pgfsetcolor{\tikz@pattern@color}
\pgfsetlinewidth{\mcThickness}
\pgfpathmoveto{\pgfpointorigin}
\pgfpathlineto{\pgfpoint{0}{\mcSize}}
\pgfusepath{stroke}
}}
\makeatother

% Pattern Info
 
\tikzset{
pattern size/.store in=\mcSize, 
pattern size = 5pt,
pattern thickness/.store in=\mcThickness, 
pattern thickness = 0.3pt,
pattern radius/.store in=\mcRadius, 
pattern radius = 1pt}
\makeatletter
\pgfutil@ifundefined{pgf@pattern@name@_w6z0ckqio}{
\makeatletter
\pgfdeclarepatternformonly[\mcRadius,\mcThickness,\mcSize]{_w6z0ckqio}
{\pgfpoint{-0.5*\mcSize}{-0.5*\mcSize}}
{\pgfpoint{0.5*\mcSize}{0.5*\mcSize}}
{\pgfpoint{\mcSize}{\mcSize}}
{
\pgfsetcolor{\tikz@pattern@color}
\pgfsetlinewidth{\mcThickness}
\pgfpathcircle\pgfpointorigin{\mcRadius}
\pgfusepath{stroke}
}}
\makeatother

% Pattern Info
 
\tikzset{
pattern size/.store in=\mcSize, 
pattern size = 5pt,
pattern thickness/.store in=\mcThickness, 
pattern thickness = 0.3pt,
pattern radius/.store in=\mcRadius, 
pattern radius = 1pt}
\makeatletter
\pgfutil@ifundefined{pgf@pattern@name@_xyst2nhwr lines}{
\pgfdeclarepatternformonly[\mcThickness,\mcSize]{_xyst2nhwr}
{\pgfqpoint{0pt}{0pt}}
{\pgfpoint{\mcSize+\mcThickness}{\mcSize+\mcThickness}}
{\pgfpoint{\mcSize}{\mcSize}}
{\pgfsetcolor{\tikz@pattern@color}
\pgfsetlinewidth{\mcThickness}
\pgfpathmoveto{\pgfpointorigin}
\pgfpathlineto{\pgfpoint{\mcSize}{0}}
\pgfusepath{stroke}}}
\makeatother

% Pattern Info
 
\tikzset{
pattern size/.store in=\mcSize, 
pattern size = 5pt,
pattern thickness/.store in=\mcThickness, 
pattern thickness = 0.3pt,
pattern radius/.store in=\mcRadius, 
pattern radius = 1pt}
\makeatletter
\pgfutil@ifundefined{pgf@pattern@name@_303clpc4i}{
\pgfdeclarepatternformonly[\mcThickness,\mcSize]{_303clpc4i}
{\pgfqpoint{0pt}{0pt}}
{\pgfpoint{\mcSize+\mcThickness}{\mcSize+\mcThickness}}
{\pgfpoint{\mcSize}{\mcSize}}
{
\pgfsetcolor{\tikz@pattern@color}
\pgfsetlinewidth{\mcThickness}
\pgfpathmoveto{\pgfqpoint{0pt}{0pt}}
\pgfpathlineto{\pgfpoint{\mcSize+\mcThickness}{\mcSize+\mcThickness}}
\pgfusepath{stroke}
}}
\makeatother

% Pattern Info
 
\tikzset{
pattern size/.store in=\mcSize, 
pattern size = 5pt,
pattern thickness/.store in=\mcThickness, 
pattern thickness = 0.3pt,
pattern radius/.store in=\mcRadius, 
pattern radius = 1pt}
\makeatletter
\pgfutil@ifundefined{pgf@pattern@name@_whp56n1nz}{
\pgfdeclarepatternformonly[\mcThickness,\mcSize]{_whp56n1nz}
{\pgfqpoint{0pt}{-\mcThickness}}
{\pgfpoint{\mcSize}{\mcSize}}
{\pgfpoint{\mcSize}{\mcSize}}
{
\pgfsetcolor{\tikz@pattern@color}
\pgfsetlinewidth{\mcThickness}
\pgfpathmoveto{\pgfqpoint{0pt}{\mcSize}}
\pgfpathlineto{\pgfpoint{\mcSize+\mcThickness}{-\mcThickness}}
\pgfusepath{stroke}
}}
\makeatother

% Pattern Info
 
\tikzset{
pattern size/.store in=\mcSize, 
pattern size = 5pt,
pattern thickness/.store in=\mcThickness, 
pattern thickness = 0.3pt,
pattern radius/.store in=\mcRadius, 
pattern radius = 1pt}
\makeatletter
\pgfutil@ifundefined{pgf@pattern@name@_n8m5z0zvy}{
\pgfdeclarepatternformonly[\mcThickness,\mcSize]{_n8m5z0zvy}
{\pgfqpoint{0pt}{0pt}}
{\pgfpoint{\mcSize}{\mcSize}}
{\pgfpoint{\mcSize}{\mcSize}}
{
\pgfsetcolor{\tikz@pattern@color}
\pgfsetlinewidth{\mcThickness}
\pgfpathmoveto{\pgfqpoint{0pt}{\mcSize}}
\pgfpathlineto{\pgfpoint{\mcSize+\mcThickness}{-\mcThickness}}
\pgfpathmoveto{\pgfqpoint{0pt}{0pt}}
\pgfpathlineto{\pgfpoint{\mcSize+\mcThickness}{\mcSize+\mcThickness}}
\pgfusepath{stroke}
}}
\makeatother
\tikzset{every picture/.style={line width=0.75pt}} %set default line width to 0.75pt        

\begin{tikzpicture}[x=1pt,y=1pt,yscale=-1,xscale=1]
%uncomment if require: \path (0,336); %set diagram left start at 0, and has height of 336

%Shape: Rectangle [id:dp3820356906332222] 
\draw  [pattern=_abi33ot3k,pattern size=6pt,pattern thickness=0.75pt,pattern radius=0pt, pattern color={rgb, 255:red, 0; green, 0; blue, 0}] (100,120) -- (140,120) -- (140,160) -- (100,160) -- cycle ;
%Shape: Rectangle [id:dp11458860148111383] 
\draw  [pattern=_w6z0ckqio,pattern size=6pt,pattern thickness=0.75pt,pattern radius=0.75pt, pattern color={rgb, 255:red, 0; green, 0; blue, 0}] (140,120) -- (180,120) -- (180,160) -- (140,160) -- cycle ;
%Shape: Rectangle [id:dp9037311824173566] 
\draw  [pattern=_xyst2nhwr,pattern size=6pt,pattern thickness=0.75pt,pattern radius=0pt, pattern color={rgb, 255:red, 0; green, 0; blue, 0}] (100,160) -- (140,160) -- (140,200) -- (100,200) -- cycle ;
%Shape: Rectangle [id:dp527989031212017] 
\draw  [pattern=_303clpc4i,pattern size=6pt,pattern thickness=0.75pt,pattern radius=0pt, pattern color={rgb, 255:red, 0; green, 0; blue, 0}] (140,160) -- (180,160) -- (180,200) -- (140,200) -- cycle ;
%Shape: Rectangle [id:dp529062958693298] 
\draw  [pattern=_whp56n1nz,pattern size=6pt,pattern thickness=0.75pt,pattern radius=0pt, pattern color={rgb, 255:red, 0; green, 0; blue, 0}] (125,145) -- (155,145) -- (155,175) -- (125,175) -- cycle ;
%Shape: Path Data [id:dp0033321936188244194] 
\draw  [pattern=_n8m5z0zvy,pattern size=6pt,pattern thickness=0.75pt,pattern radius=0pt, pattern color={rgb, 255:red, 0; green, 0; blue, 0}] (300,120) -- (300,145) -- (285,145) -- (285,160) -- (260,160) -- (260,120) -- (300,120) -- cycle ;
%Shape: Rectangle [id:dp09832103285945681] 
\draw   (300,120) -- (340,120) -- (340,160) -- (300,160) -- cycle ;
%Shape: Rectangle [id:dp18453169269465775] 
\draw   (260,160) -- (300,160) -- (300,200) -- (260,200) -- cycle ;
%Shape: Rectangle [id:dp4996391879370258] 
\draw   (300,160) -- (340,160) -- (340,200) -- (300,200) -- cycle ;
%Shape: Rectangle [id:dp13305819716334977] 
\draw   (285,145) -- (315,145) -- (315,175) -- (285,175) -- cycle ;
%Right Arrow [id:dp16608511172284524] 
\draw   (210,153.35) -- (222,153.35) -- (222,150) -- (230,156.7) -- (222,163.4) -- (222,160.05) -- (210,160.05) -- cycle ;

% Text Node
\draw (235,102.4) node [anchor=north west][inner sep=0.75pt]  [font=\footnotesize]  {$ \begin{array}{l}
\Delta_\alpha( \textbf{x}_{1})\\
\end{array}$};

\end{tikzpicture}

    \caption{Example \ref{exnetinterval}. }
    \label{firstexample}
\end{figure}

\end{ej}

\begin{obs} From the definition we get
\begin{itemize}
    \item $K\subset\bigcup_{B\in \F_\delta}B$
    \item If $B,D\in \F_\delta$ then either $B=D$ or $\inte(B)\cap \inte(D)=\emptyset$.
    \item If $\delta>\delta'$ then for $B'\in \F_{\delta'}$ there is a unique $B\in \F_\delta$ such that $B'\subset B$.
\end{itemize}
\end{obs}
We note that the classical definition of net interval in $\R$ is given in terms of boundary points. The problem with this approach in $\R^d$ is that there is no clear order to follow. Moreover, while these definitions are not equivalent, they coincide when the IFS is equicontractive. When the IFS is not equicontractive, our way of defining the net interval might result in net intervals that are the union of the classical net intervals, as we observe in the next example.
\begin{ej}\label{exampenetintervaldif}
 
Consider the IFS $\cS=\{S_i\}_{i=0}^4$ given by $S_1(x)= x/3-1/3,S_2(x)= x/3,S_3(x) x/3+1/3,S_4(x)= x/9$. In particular, we have $S_1([-1/2,1/2])=[-1/2,-1/6]$, $S_2([-1/2,1/2])=[-1/6,1/6]$, $S_1([-1/2,1/2])=[1/6,1/2]$ and $S_1([-1/2,1/2])=[-1/18,1/18]$. For $1/3<\alpha<1$, we have the net intervals $\Delta_1=[-1/2,-1/6]$, $\Delta_2=[-1/6,-1/18]\cup[1/18,1/6]$, $\Delta_3=[-1/6,1/6]$, and $\Delta_4=[1/6,1/2]$. The classical definition of net interval would divide $\Delta_2$ into two net intervals. The family of classical net intervals is $\Delta_1=[-1/2,-1/6]$, $\Delta_2^a=[-1/6,-1/18]$, $\Delta_2^b[1/18,1/6]$, $\Delta_3=[-1/6,1/6]$, and $\Delta_4=[1/6,1/2]$. In Figure \ref{diffnetintervals} we can see the net intervals of this example as presented with our definition.

\begin{figure}[h]
    \centering

\tikzset{every picture/.style={line width=0.75pt}} %set default line width to 0.75pt        

\begin{tikzpicture}[x=0.75pt,y=0.75pt,yscale=-1,xscale=1]
%uncomment if require: \path (0,336); %set diagram left start at 0, and has height of 336

%Straight Lines [id:da501605349756472] 
\draw    (100,100) -- (190,100) ;
%Straight Lines [id:da21870141544505617] 
\draw    (280,100) -- (370,100) ;
%Straight Lines [id:da12560593398477293] 
\draw    (190,110) -- (280,110) ;
%Straight Lines [id:da8969313665390141] 
\draw    (220,120) -- (250,120) ;
%Straight Lines [id:da8041016241000298] 
\draw    (100,150) -- (190,150) ;
%Straight Lines [id:da7639191698763702] 
\draw    (280,150) -- (370,150) ;
%Straight Lines [id:da4136972951483894] 
\draw    (220,170) -- (250,170) ;
%Straight Lines [id:da6014144694047978] 
\draw    (190,160) -- (220,160) ;
%Straight Lines [id:da5937240915709557] 
\draw    (250,160) -- (280,160) ;
%Straight Lines [id:da2744891731087533] 
\draw    (205,160) -- (228.14,150.74) ;
\draw [shift={(230,150)}, rotate = 158.2] [color={rgb, 255:red, 0; green, 0; blue, 0 }  ][line width=0.75]    (10.93,-3.29) .. controls (6.95,-1.4) and (3.31,-0.3) .. (0,0) .. controls (3.31,0.3) and (6.95,1.4) .. (10.93,3.29)   ;
%Straight Lines [id:da9139447765368401] 
\draw    (265,160) -- (241.86,150.74) ;
\draw [shift={(240,150)}, rotate = 21.8] [color={rgb, 255:red, 0; green, 0; blue, 0 }  ][line width=0.75]    (10.93,-3.29) .. controls (6.95,-1.4) and (3.31,-0.3) .. (0,0) .. controls (3.31,0.3) and (6.95,1.4) .. (10.93,3.29)   ;

% Text Node
\draw (111,82.4) node [anchor=north west][inner sep=0.75pt]  [font=\tiny]  {$S_{1}([ -1/2,1/2])$};
% Text Node
\draw (200,92.4) node [anchor=north west][inner sep=0.75pt]  [font=\tiny]  {$S_{2}([ -1/2,1/2])$};
% Text Node
\draw (290,82.4) node [anchor=north west][inner sep=0.75pt]  [font=\tiny]  {$S_{3}([ -1/2,1/2])$};
% Text Node
\draw (260,121.4) node [anchor=north west][inner sep=0.75pt]  [font=\tiny]  {$S_{4}([ -1/2,1/2])$};
% Text Node
\draw (131,141) node [anchor=north west][inner sep=0.75pt]  [font=\tiny]  {$\Delta _{1}$};
% Text Node
\draw (231,171.4) node [anchor=north west][inner sep=0.75pt]  [font=\tiny]  {$\Delta _{3}$};
% Text Node
\draw (317,141) node [anchor=north west][inner sep=0.75pt]  [font=\tiny]  {$\Delta _{4}$};
% Text Node
\draw (229,141) node [anchor=north west][inner sep=0.75pt]  [font=\tiny]  {$\Delta _{2}$};
% Text Node
\draw (440,105) node   [align=left] {\begin{minipage}[lt]{68pt}\setlength\topsep{0pt}
Level $\displaystyle \alpha $
\end{minipage}};
% Text Node
\draw (430,145) node   [align=left] {\begin{minipage}[lt]{68pt}\setlength\topsep{0pt}
Net intervals at level $\displaystyle \alpha $
\end{minipage}};

\end{tikzpicture}

    \caption{Net intervals. Example \ref{exampenetintervaldif}}
    \label{diffnetintervals}
\end{figure}
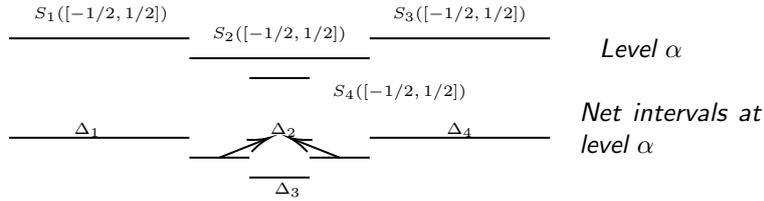

Although the net intervals as defined here and in \cite{hare/hare/rutar} differ, the difference is not significant. As we have mention when the IFS are equicontractive the two notions coincide. If the IFS is not equicontractve we have Theorem \ref{thm:FNC vs classical FNC}.

\end{ej}

We will now introduce how we will normalize the net intervals. We will require the following notation; Given $L\in \R^+$ and $e\in \R^d$, we define $f_L^e:\R^d\to \R^d$ with $f_L^e(x)=Lx+e$.  We fix the next normalization for net intervals.
\begin{defi}

    Given $\Delta\in\subset \F$, let $L\in \R^+$ be the largest real number with the properties that there is $e\in\R^d$ such that $f_L^e([0,1]^d)\subset \Delta$. Then the normalize size $m(\Delta)=L$. We note that the set of $e\in \R^d$ is a close subset of $\Delta$. We let $a$ be the smallest of such $e$(order $\R^d$ with lexicographical order). We set \begin{equation}
        T_\Delta(x)=f_L^a(x)=Lx+a
    \end{equation}

    We say $T(x)=Lx+a$ is a neighbor of $\Delta\in \F_\alpha$  if there exist $\sigma\in \Lambda_\alpha$ such that $S_\sigma(\K)\supset \Delta$ and $T=T_\Delta^{-1}\circ S_\sigma$. In this case, we also say that $S_\sigma$ generates the neighbor $T$. The \emph{neighbor set} of $\Delta$ is the maximal \begin{equation}
        V_\alpha(\Delta)=\{T_1,\dots,T_m\}.
    \end{equation}
    Here, each $T_i=T_\Delta^{-1}\circ S_{\sigma_i}$ is a distinct neighbor of $\Delta$. When the generation is clear, we write $V(\Delta)$.
\end{defi}
Note that $m(\Delta)$ is determined by the length of the side of the largest hypercube that is correctly oriented, which can fit inside of $\Delta$. 
\begin{defi}
    We say that an IFS satisfies the \emph{finite neighbor condition} (FNC) if there are only finitely many neighbor sets.
\end{defi}

In practice, we identify the net intervals with their neighbor set. In other words, in Example \ref{exampenetintervaldif} we would identify $\Delta_1$, $\Delta_4$ as the same net interval. This is because the key information is on the neighbor set. 

We note that many of this concepts we introduce here, have parallels to those introduced in \cite{hare/hare/rutar}. To differentiate them, we will add the word \emph{classical} to the concepts from \cite{hare/hare/rutar}. In other words, classical net intervals are net intervals as introduced in \cite{hare/hare/rutar}. We also recall that in \cite{hare/hare/rutar} the work was done for IFS in $\R$. Hence, any concept that has the word classical is for IFS in $\R$.
\section{Characterization of the different overlap conditions}\label{SCharact}
In this section, we give some characterization and properties of the overlaps conditions. We start with the FNC. Although we do not get an equivalence between the FNC and the GFTC\textsubscript{co}, we show that the FNC implies the GFTC\textsubscript{co}. We compare the FNC we propose to the one given in \cite{hare/hare/rutar}, we show the two notions are equivalent in $\R$. We give characterizations of the GFTC when the IFS is equicocantrive, and we close this section by characterizing the WSC in $\R^d$, this is the same known characterization that was obtain in $\R$ in \cite{hare/hare/rutar}.
\begin{thm}\label{thm3.4}
    If an IFS satisfies the FNC, then it satisfies the GFTC\textsubscript{co}.
\end{thm}
\begin{proof}
    Set,
    \[\E= \E_\cS(V)=\bigcup_{\alpha>0}\left\{S_\sigma^{-1}\circ S_\tau:\sigma,\tau\in\Lambda_\alpha,S_\sigma(V)\cap S_\tau(V)\neq\emptyset\right\}.\]
    Then what we want to prove is that the statement ``$\cS$ has finitely many neighbors sets implies $\E$ is finite''. 

    Suppose the IFS $\cS$ has finitely many neighbor sets. Let $\sigma, \tau\in \Lambda_\alpha$ be arbitrary such that$I:=S_\sigma(\K)\cap S_\tau(\K)\neq \emptyset$. Then there is a net interval $\Delta\in \F$ contained in $I$, so that $S_\sigma$ and $S_\tau$ generate neighbors of $\Delta$. In particular, $T_\Delta^{-1}\circ S_\sigma$ and $T_\Delta^{-1}S_\tau$ must be two of the finitely many neighbors.  Hence
    \[S_\sigma^{-1}\circ S_\tau=S_\sigma\circ T_\Delta\circ T_\Delta^{-1}\circ S_\tau=(T_\Delta^{-1}\circ S_\sigma)^{-1}\circ(T_\Delta^{-1}\circ S_\tau)\]
    can only take finitely many values. So $\E$ is finite.

\end{proof}
\begin{remark}
    Theorem \ref{thm3.4} becomes an equivalence when changing the FNC for the classical FNC, see \cite[Theorem~3.4]{hare/hare/rutar}. 
\end{remark}
\begin{thm}\label{thm:FNC vs classical FNC}
    An IFS in $\R$ satisfies the FNC if and only if it satisfies the classical FNC. 
\end{thm}
\begin{proof}
    This follows from Theorem \ref{thm3.4}, \cite[Theorem~3.4]{hare/hare/rutar} for one implication. The other implication follows from the fact that our net intervals are at most a finite union of classical net intervals.
\end{proof}

\begin{lem}\label{gencharactgftc}
    Let $d_\ell$ be a metric on $\R^d$ and $\{S_i\}$ be an IFS with respect to $d_\ell$ such that $S_i(x)=r A_ix+d_i$. In addition, assume that the attractor $K\subset \operatorname{cl}(B_\ell(0,1/2))$, $0<r<1$, and $A_i$ generate a finite group. Then $\{S_i\}$ satisfies the GFTC\textsubscript{co} with $V=B_\ell(0,1/2)$ if and only if there exists a finite set $F$ such that for $|\sigma|=|\tau|=n$ either
    \begin{equation}
        r^{-n}d_\ell(S_\sigma(0),S_\tau(0))>1 \text{ or } r^{-n}(S_\sigma(0)-S_\tau(0))\in F.
    \end{equation}
\end{lem}
Here $B_\ell(0,1/2)$ is the ball center at $0$ of radius $1/2$ with respect to $d_\ell$. 
\begin{proof}
       
    Note that any point $y\in S_\sigma(V)$ has the property that $d_\ell(S_\sigma(0),y)\leq \frac{r^n}{2}$, similar to $x\in S_\tau(V)$. Hence,
    \begin{align*}
        d_\ell(S_\sigma(0),S_\tau(0))&\leq r^n\\
        r^{-n}d_\ell(S_\sigma(0),S_\tau(0))&\leq 1.
    \end{align*}

    If $r^{-n}d_\ell(S_\sigma(0),S_\tau(0))>1$, then we have $S_\sigma(V)\cap S_\tau(V)=\emptyset$ so $S_\sigma^{-1}\circ S_\tau \notin N(S_\sigma)$.

    Assume $r^{-n}d_\ell(S_\sigma(0),S_\tau(0))\leq1$, and $S_\sigma(V)\cap S_\tau(V)\neq\emptyset$. Hence 
    \[S_\sigma^{-1}\circ S_\tau\in N(S_\sigma)\text{ and }r^{-n}(S_\sigma(0)-S_\tau(0))\in F.\]
    Note that $S_\sigma(x)=r^n\prod_i A_i x +\sum_i (\prod_j A_jd_ir^{i-1})$ and $S_\tau(x)=r^n\prod_i B_i x +\sum_i (\prod_j B_jb_ir^{i-1})$. Set $\overline{A}_j = \prod_{i=1}^j A_i$, and similarly $\overline{B}_j$. Hence
    \begin{align*}
        S_\sigma^{-1}&\circ S_\tau(x)=r^{-n}\left(\overline{A}_n\right)^{-1}\left(r^n\left(\overline{B}_n\right) x +\sum_i \left(\left(\overline{B}_i\right)b_ir^{i-1}\right) -\sum_i \left(\left(\overline{A}_i\right)d_ir^{i-1}\right)\right)\\
        &=\left(\overline{A}_n\right)^{-1}\left(\overline{B}_n\right) x +\left(\overline{A}_n\right)^{-1}r^{-n}\left(\sum_i \left(\left(\overline{B}_i\right)b_ir^{i-1}\right)-\sum_i \left(\left(\overline{A}_i\right)d_ir^{i-1}\right)\right)\\
        &=\left(\overline{A}_n\right)^{-1}\left(\overline{B}_n\right) x +\left(\overline{A}_n\right)^{-1}r^{-n}\left(S_\tau(0)-S_\sigma(0)\right).
    \end{align*}
    
Since the $A_i$ generate a finite group. Then $F$ is finite if and only if there are finitely many $S_\sigma^{-1}\circ S_\tau(x)$ such that $S_\sigma(V)\cap S_\tau(V)\neq \emptyset$.

\end{proof}
\begin{remark}
    It is not hard to come up with examples where $A_i$ do not generate a finite group and the IFS has the GFTC or even the OSC. In most of these cases, what is happening is that there is a lack of interaction between the different similarities involved.
\end{remark}
\begin{lem}\label{equivofnorms}
        Let $\cS=\{S_i\}_{i=0}^n$ be an IFS with respect to metrics $d_\ell$ and $d_\ell'$ on $\R^d$. If $\cS$ satisfies GFTC with respect to $d_\ell$ with the invariant set $V$, then it is GFTC with respect to $d_\ell'$ with the same open set $V$. Similarly for GFTC\textsubscript{co}.
\end{lem}
\begin{proof}
    The proof is a direct consequence that the metrics $d_\ell$ and $d_\ell'$ are equivalent in $\R^d$.
\end{proof}

\begin{lem}\label{supnorm}
    Let $\{S_i\}$ be an IFS such that $S_i(x)=rA_i x+d_i$,with $0<r<1$ and attractor $K\subset \K$. Further, assume that the $A_i$ generate a finite group. Then $\{S_i\}$ satisfies the GFTC with $V=(-1/2,1/2)^d$ if and only if there exists a finite set $F$ such that for $|\sigma|=|\tau|=n$ either
    \begin{equation}
        r^{-n}\|S_\sigma(0)-S_\tau(0)\|_\infty>1 \text{ or } r^{-n}(S_\sigma(0)-S_\tau(0))\in F.
    \end{equation}
\end{lem}
\begin{proof}
The proof is a consequence of Lemma \ref{gencharactgftc} and Lemma \ref{equivofnorms}.
\end{proof}
The next corollary is a direct application of Lemma \ref{supnorm}, but since it is a useful case, we state it none the less.
\begin{cor}
        Let $\{S_i\}$ be an IFS such that $S_i(x)=r x+d_i$, with $0<r<1$ and attractor $K\subset \K$. Then $\{S_i\}$ satisfies the GFTC with $V=(-1/2,1/2)^d$ if and only if there exists a finite set $F$ such that for $|\sigma|=|\tau|=n$ either
    \begin{equation}
        r^{-n}\|S_\sigma(0)-S_\tau(0)\|_\infty>1 \text{ or } r^{-n}(S_\sigma(0)-S_\tau(0))\in F.
    \end{equation}
\end{cor}
\begin{prop}\label{prop2.8}
    An IFS has the weak separation condition if and only if 
    \begin{equation}
        \sup_{\Delta\in \F}\# V(\Delta)<\infty.
    \end{equation}
\end{prop}
\begin{proof}
    Find the bounds $N_1,\dots, N_{2^d}$ from the Definition \ref{def 2.2} for $x_i=v_1,\dots, v_{2^d}$ (the corners of the hypercube) and set $\tau=\epsilon$ as the empty word. Let $\Delta\in \F_\alpha$ be a net interval, set
    \[E=\{S_\sigma(v_i):T_\Delta^{-1}\circ S_\sigma\in V(\Delta)\}.\]
    Note that $\#V(\Delta)\leq\#\mathcal{P}(E)$ since each neighbor of $\Delta$ corresponds to a finite collection of points from $E$. Thus, it is enough to show that $E$ is finite. 

    Since $\Delta\subset S_\sigma([-1/2,1/2]^d)$ for some $\sigma$, there is a hypercube $A_1$ with side $\alpha$ such that $\Delta\subset A$. Consider $A_2,\dots, A_{3^d}$ the cubes that surround $A$. Then $\Delta\subset \cup_{i=1}^{3^d}A_i=:I$. Then $I$ contains at most $3^dN_i$ distinct points of the form $S_\sigma (v_i)$. Thus
    \[\#E\leq 3^d\left(\sum_i N_i\right).\]

    Conversely, suppose that
    \[\sup_\Delta\#V(A)=M<\infty.\]
    Fix an arbitrary generation $\alpha$ and a closed hypercube $I$ with side $\alpha$. Fix $x_0\in [-1/2,1/2]^d$. Let $J$ be a closed hypercube such that it has side $2\alpha$ and with the same center and orientation as $I$. Assume $x=S_\sigma (x_0)\in I \subset J$ for $\sigma \in \Lambda_\alpha$. If $I$ contains $N$ distinct points of the form $S_\sigma(x_0)$ for $\sigma \in \Lambda_\alpha$, setting 
    \[Y=\{S_\sigma:S_\sigma (x_0)\in I,\sigma\in \Lambda_\alpha\}.\]
    We see that $\#Y\geq N$. Thus
    \begin{equation}\label{eq1prop2.2}
    \sum_{f\in Y}m(f([-1/2,1/2]^d))\geq N\alpha r_{\min}.
    \end{equation}
    Since $\sup_\Delta\#V(A)=M$, any point $x\in J$ can be contained in the interior of at most $M$ sets of the form $f([-1/2,1/2]^d)$ with $f\in Y$. Hence
    \begin{equation}\label{eq2prop2.2}
        \sum_{f\in Y}m(f([-1/2,1/2]^d))\leq M m(j)\leq M 2\alpha.
    \end{equation}
    Combining \eqref{eq1prop2.2} and \eqref{eq2prop2.2}, we have \[N\alpha r_{\min}\leq M2\alpha\text{ or }N\leq\frac{2M}{r_{\min}}.\]
\end{proof}
If there are finitely many neighbor sets, then $\sup_{\Delta\in \F}\#V(\Delta)<\infty$. Hence, by Proposition \ref{prop2.8} we have
\begin{cor}
    Any IFS satisfying the FNC has the WSC.
\end{cor}
\section{Weak Separation Condition implies Finite Neighbor Condition}\label{SEquiv}
In this section, the main result of this paper is presented. We start with two technical lemmas. We then proceed to prove our main result of Theorem \ref{thm4.4}. 

The proof of the equivalence is done in two steps. The first step is the lemma \ref{lem4.2}. In Lemma \ref{lem4.2} we show that for any $\delta>0$ there are finitely many overlaps that are greater than $\delta$. Step 2 is Theorem \ref{thm4.4}, where we show that the WSC implies that there is a lower bound on how small an overlap can be, which will allow us to bound $\E_\cS(V)$.

For the rest of this section we say that a matrix $A$ fixes $\K$ if, when considering the function $f(x)=Ax$, then $f(\K)=\K$.
\begin{lem}\label{lem4.1}
    \Assumptions.
    
Fix $\delta>0$. There exists a constant $C=C(\delta)$ such that for any $\alpha>0$ and $\sigma,\tau\in \Lambda_\alpha$ with $m(S_\sigma(\K)\cap S_\tau(\K))>\delta\alpha$, there is some word $\phi$ with $|\phi|\geq C$ and a choice of $\psi\in\{\sigma,\tau\}$ such that \[S_{\psi\phi}(\K)\subset S_\sigma(\K)\cap S_\tau(\K).\]
\end{lem}
\begin{proof}
    Let $I=S_\sigma(\K)\cap S_\tau(\K)$. We note that $I=[a_1,b_1]\times[a_2,b_2]\times\dots\times[a_d,b_d]$. Without loss of generality, we assume that one corner of $I$ is of the form $S_\sigma(v)$ where $v$ is a corner of $[-1/2,1/2]^d$. Put $\psi=\sigma$.
    
Let $C=C(\delta)=\delta(r_{\min})^2$. Let $\phi$ be such that $S_\phi\in\Lambda_\delta$ and $v\in S_\phi(\K)$. Note that $r_\phi\geq C$ and $S_{\psi\phi}(\K)$ is a hypercube with side less than $\delta\alpha$. Since $v\in S_{\psi\phi}(\K)\subset\K$, we conclude that \[S_{\psi\phi}(\K)\subset S_\sigma(\K)\cap S_\tau(\K).\]
\end{proof}
\begin{lem}\label{lem4.2}
   \Assumptions. If $\{S_i\}$  satisfies the weak separation condition, then for each $\delta>0$ there exists a finite set $\E_\delta$ such that for any generation $\alpha>0$ and $\sigma,\tau \in \Lambda_\alpha$ either
\[m(S_\sigma([-1/2,1/2]^d\cap S_\tau([-1/2,1/2]^d)<\delta\alpha\text{ or }S_\sigma^{-1}\circ S_\tau\in \E_\delta.\]
\end{lem}
\begin{proof}
 Fix $\delta>0$. Choose a net interval $\Delta_0$ with a maximum number of neighbors and assume $\Delta\in \F_\beta$, Proposition \ref{prop2.8} guarantees that this is possible.

 Note that since we have full support, there is a constant $C=C(\delta)$ as in Lemma \ref{lem4.1}. Define \[\Gamma=\{T\circ S_\psi^{-1}:\psi\in \Sigma^*, r_\psi\geq Cr_{\min}^2, T\in V(\Delta_0)\},\]
and put
\[\E_\delta=\{f^{-1}\circ g:f,g\in \Gamma\},\]
we see that $\E_\delta$ is finite since $\Gamma$ is finite. 

Let $\sigma,\tau \in \Lambda_\alpha$ be arbitrary with $m(S_\sigma([-1/2,1/2]^d)\cap S_\tau([-1/2,1/2]^d))\geq \delta \alpha$. Choose $\psi$ and $\phi$, such that $r_\phi>C$, $\psi\in \{\sigma,\tau\}$, and
\[S_{\psi\phi}([-1/2,1/2]^d)\subset S_\sigma([-1/2,1/2]^d)\cap S_\tau([-1/2,1/2]^d).\]
Without loss of generality, assume $\psi=\sigma$. Set 
\[\gamma=r_\sigma r_\phi \beta.\]

We claim that the interval $\Delta_1=S_{\sigma\phi}(\Delta_0)$ is a net interval of generation $\gamma$ with $V(\Delta_0)=V(\Delta_1)$.

To see the claim, let $\Delta_0$ have neighbors generated by $S_{\omega_1},\dots,S_{\omega_m}$ with $\omega_i\in \Lambda_\beta$. Hence $\{\sigma\phi\omega_i\}_{i=1}^m$ are words of generation $\Lambda_\gamma$. Note that $(\mathrm{int}(\Delta_1)\cap K)\neq\emptyset$. 

In particular, if $\Delta_1\notin \F_\gamma$, then there exist some $\tau\in \lambda_\gamma$ such that $S_\tau\notin\{S_{\sigma\phi\omega_i}\}_{i=1}^m$ and $S_\tau([-1/2,1/2]^d)\cap(\mathrm{int}(\Delta_1)\cap K)\neq\emptyset$. Then there exist $\Delta_2\in F_\gamma$ with $\Delta_2\subset \Delta_1\cap S_\tau([-1/2,1/2]^d)$ where $\Delta_2$ has neighbors generated by $\{\omega_1,\dots,\omega_m\}\cup\{\tau\}$ contradicting the maximality of $m$. Thus, $\Delta_1=\Delta_2$ and $\Delta_1\in F_\gamma$ with neighbors generated by $\sigma\phi\omega_i$. Note that $T_{\Delta_1}=r_{\sigma\phi}\cdot T_{\Delta_0}+d_{\sigma\phi}$, where $d_{\sigma\phi}$ refers to the constant of $S_{\sigma\phi}(x)=r_{\sigma\phi}A_{\sigma\phi}x+d_{\sigma\phi}$. Hence
\[V(\Delta_1)=\{T_{\Delta_1}^{-1}\circ S_{\sigma\phi\omega_i}\}_{i=1}^m=\{T_{\Delta_0}^{-1}\circ S_{\sigma \phi}^{-1}\circ S_{\omega_i}\}_{i=1}^m=V(\Delta_0).\]

Now we will show that $S_\sigma^{-1}\circ S_\tau\in \E_\delta$. Since $K=[-1/2,1/2]^d$ and $\Delta_1\in S_\sigma([-1/2,1/2]^d)\cap S_\tau([-1/2,1/2]^d)$, the words $\sigma$ and $\tau$ must be prefixes of $\xi_1,\xi_2\in \Lambda_\gamma$ that generate neighbors $T_1, T_2$ of $\Delta_1$, respectively. Let $\xi_1=\sigma \psi_1$ and $\xi_2=\tau\psi_2$. Since $\xi_1,\xi_2\in \Lambda_\gamma$ and $\sigma,\tau\in \Lambda_\alpha$ we have for each $i=1,2$ that  
\begin{equation}\label{eq4.2}
    r_{\psi_i}\geq \frac{\gamma}{\alpha}r_{\min}\geq \frac{\alpha r_\phi\beta}{\alpha}r_{\min}^2\geq C\beta r_{\min}^2.
\end{equation}
Since $T_{\Delta_1}^{-1}\circ S_{\xi_i}=T_i$ we have
\begin{align*}
    S_\sigma^{-1}\circ S_\tau&=S_{\psi_1}\circ (S_{\xi_1}^{-1}\circ S_{\xi_2})\circ S_{\psi_2}^{-1}\\
    &=S_{\psi_1}\circ (T_1^{-1}\circ T_{\Delta_1}^{-1}\circ T_{\Delta_1}\circ T_{2})\circ S_{\psi_2}^{-1}\\
    &=(T_1\circ S_{\psi_1}^{-1})^{-1}\circ  (T_{2}\circ S_{\psi_2}^{-1}).
\end{align*}
This is an element of $\E_\delta$ by \eqref{eq4.2}.
\end{proof}

\begin{thm}\label{thm4.4}
    \Assumptions. If $\cS$ satisfies the weak separation property and the self-similar set $K=[-1/2,1/2]^d$, then $\cS$ has the GFTC\textsubscript{co}.
\end{thm}
\begin{proof}
    Assume  $\cS=\{S_i\}_{i=1}^k$. Set 
    \[\delta=r_{\min}\min\{|0-S_i(0)|_j: 1\leq i\leq k, 0\neq S_i^j(0)\}>0.\]
    Let $\E_\delta$ be the corresponding finite set  as in Lemma \ref{lem4.2}. Put
    \[\cG=\{g^{-1}\circ f\circ h: f\in \E_\delta, g,h\in\{Id, S_1,\dots, S_k\}.\]
    Note that $\cG$ is a finite set. We may now define 
\begin{multline*}
     \varepsilon_1:=\min\{m(S_\phi([-1/2,1/2]^d\cap S_\psi([-1/2,1/2]^d):\\r_\phi,r_\psi\geq r_{\min}^2,m(S_\phi([-1/2,1/2]^d\cap S_\psi([-1/2,1/2]^d)\neq \emptyset\}
\end{multline*}
and
 \begin{multline*}
   \varepsilon_2:=\min\{m([-1/2,1/2]^d\cap f([-1/2,1/2]^d):\\ f\in \cG\text{ and }[-1/2,1/2]^d\cap f([-1/2,1/2]^d)\neq \emptyset\}  
 \end{multline*}   
    
    Fix \[0<\varepsilon\leq\min\{\varepsilon_1,r_{\min}\varepsilon_2\}\]
    Note that $\varepsilon\leq r_{\min}$.

    We claim that for any $\alpha>0$ and $\sigma,\tau\in \Lambda_\alpha$ with $S_\sigma([-1/2,1/2]^d)\cap S_\tau([-1/2,1/2]^d)\neq \emptyset$ we have, 
    \[m(S_\sigma([-1/2,1/2]^d)\cap S_\tau([-1/2,1/2]^d))\geq\varepsilon\alpha.\]

    Once the claim is verified, we are done since Lemma \ref{lem4.2} will imply $\E_{\cS}\subset \E_\varepsilon$.

    We will now prove the claim by contradiction. Assume the claim is false. Then there exists $0<\alpha\leq 1$ and $\sigma,\tau \in \Lambda_\alpha$ such that $S_\sigma([-1/2,1/2]^d)\cap S_\tau([-1/2,1/2]^d)\neq \emptyset$, and
    \begin{equation}\label{eq4.3}
        m(S_\sigma([-1/2,1/2]^d)\cap S_\tau([-1/2,1/2]^d))<\varepsilon\alpha.
    \end{equation}
    Choose $\alpha$ maximal with this property. Observe that the choice $\varepsilon\leq\varepsilon_1$ ensures that $\sigma$ and $\tau$ are both of length at least two. To see this assume that $\sigma$ had length at most $1$. Then  $\alpha\geq r_{\min}$. Consequently $r_\sigma, r_\tau\geq r_{\min}^2$ and thus 
    \[ m(S_\sigma([-1/2,1/2]^d)\cap S_\tau([-1/2,1/2]^d))\geq\varepsilon_1\geq\varepsilon\alpha.\]
    By \ref{eq4.3} this gives $\varepsilon\alpha > \varepsilon\alpha$, a contradiction. Hence $\sigma$ and $\tau$ both have length at least two.

    We let $\alpha'=\min\{r_{\sigma^-},r_{\tau^-}\}\geq \alpha$ and obtain prefixes $\sigma',\tau'$ of $\sigma$ and $\tau$ respectively, with $\sigma',\tau'\in \Lambda_{\alpha'}$. Note that $(\sigma',\tau')$ is one of $(\sigma^-,\tau),(\sigma,\tau^-)$ or $(\sigma^-,\tau^-)$.
    We first show that $ m(S_{\sigma'}([-1/2,1/2]^d)\cap S_{\tau'}([-1/2,1/2]^d))\geq \delta\alpha$. For notational simplicity we write
\begin{align*}
    S_\sigma([-1/2,1/2]^d)&=[a,b]^d=[a_1,b_1]\times[a_2,b_2]\times\dots\times[a_d,b_d],\\
     S_\tau([-1/2,1/2]^d)&=[c,e]^d=[c_1,e_1]\times[c_2,e_2]\times\dots\times[c_d,e_d],\\
      S_{\sigma'}([-1/2,1/2]^d)&=[a',b']^d=[a_1',b_1']\times[a_2,b_2]\times\dots\times[a_d',b_d'],\\
      S_{\tau'}([-1/2,1/2]^d)&=[c',e']^d=[c_1',e_1']\times[c_2',e_2']\times\dots\times[c_d',e_d'],\\
      I&=S_\sigma([-1/2,1/2]^d)\cap S_\tau([-1/2,1/2]^d).
\end{align*}
We can see an example of how to read this notation in Figure \ref{fig:Examplefornotation}. We note that the subindex indicates the direction, in other words a subindex $j$ represents the interval in direction $j$.
\begin{figure}[H]
    \centering

\tikzset{every picture/.style={line width=0.75pt}} %set default line width to 0.75pt        

\tikzset{every picture/.style={line width=0.75pt}} %set default line width to 0.75pt            

\tikzset{every picture/.style={line width=0.75pt}} %set default line width to 0.75pt        

\begin{tikzpicture}[x=0.7pt,y=0.7pt,yscale=-1,xscale=1]
%uncomment if require: \path (0,265); %set diagram left start at 0, and has height of 265

%Shape: Rectangle [id:dp9245207336914648] 
\draw   (151.33,38.87) -- (281.98,38.87) -- (281.98,167.28) -- (151.33,167.28) -- cycle ;
%Shape: Rectangle [id:dp8914003652928048] 
\draw   (229.35,103.92) -- (360,103.92) -- (360,232.33) -- (229.35,232.33) -- cycle ;

% Text Node
\draw (3,76.4) node [anchor=north west][inner sep=0.75pt]    {$ \begin{array}{l}
S_{\sigma }\left([ -1/2,1/2]^{2}\right)\\
=[ a_{1} ,b_{1}] \times [ a_{2} ,b_{2}]
\end{array}$};
% Text Node
\draw (381.33,181.4) node [anchor=north west][inner sep=0.75pt]    {$ \begin{array}{l}
S_{\tau }\left([ -1/2,1/2]^{2}\right)\\
=[ c_{1} ,e_{1}] \times [ c_{2} ,e_{2}]
\end{array}$};
% Text Node
\draw (95.33,178.4) node [anchor=north west][inner sep=0.75pt]    {$( a_{1} ,a_{2})$};
% Text Node
\draw (91.33,15.07) node [anchor=north west][inner sep=0.75pt]    {$( a_{1} ,b_{2})$};
% Text Node
\draw (178,82.73) node [anchor=north west][inner sep=0.75pt]    {$( c_{1} ,e_{2})$};
% Text Node
\draw (290.67,19.07) node [anchor=north west][inner sep=0.75pt]    {$( b_{1} ,b_{2})$};
% Text Node
\draw (247.33,125.73) node [anchor=north west][inner sep=0.75pt]    {$I$};
% Text Node
\draw (176,236.07) node [anchor=north west][inner sep=0.75pt]    {$( c_{1} ,c_{2})$};
% Text Node
\draw (283.98,170.68) node [anchor=north west][inner sep=0.75pt]    {$( b_{1} ,a_{2})$};
% Text Node
\draw (362,235.73) node [anchor=north west][inner sep=0.75pt]    {$( e_{1} ,c_{2})$};
% Text Node
\draw (362.67,82.07) node [anchor=north west][inner sep=0.75pt]    {$( e_{1} ,e_{2})$};

\end{tikzpicture}

    \caption{Example in $\R^2$ of notation in Theorem \ref{thm4.4}.}
    \label{fig:Examplefornotation}
\end{figure}
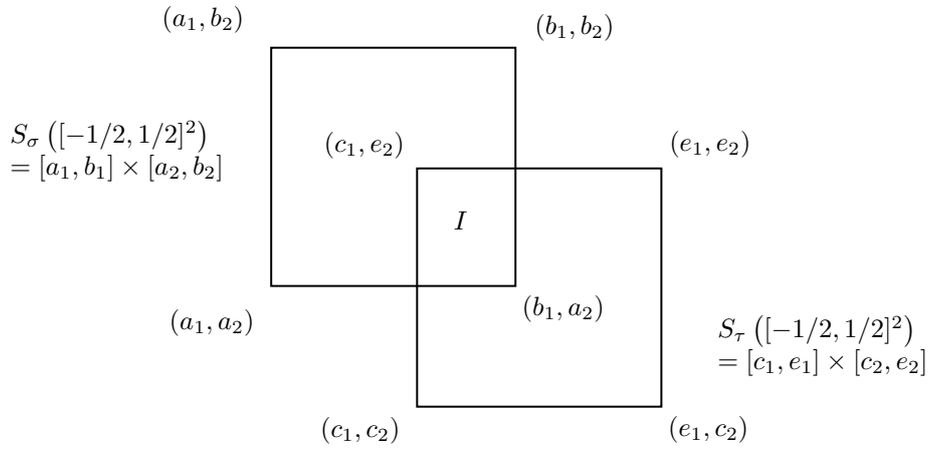
Since $\varepsilon<r_{\min}$, by swapping the roles of $\sigma$ and $\tau$ if necessary, we may assume that at least one coordinate of $I$ is of the form $[c_j,b_j]$. Otherwise (without loss of generality) $S_\sigma([-1/2,1/2]^d)\subset I$ which would imply \[m(I)\geq r_{\min}\alpha\geq \varepsilon \alpha.\]
Further assume that $b_j-c_j=m(I)$. Since $S_{\sigma'}([-1/2,1/2]^d\supset S_\sigma([-1/2,1/2]^d)$, we have $b_j'\geq b_j$ and $c_j'\leq c_j$. Moreover, by maximality of $\alpha$, we cannot have $ b_j=b_j'$ and $c_j=c_j'$.

If $b_j'>b_j$ then $\sigma'=\sigma^-$. This gives
\[b_j=(S_\sigma(u))_j\text{ and }b_j'=(S_{\sigma^-}(v))_j\]
with $u,v$ corners of the hypercube. Write $\sigma=\sigma^-i$, then \[b_j'-b_j=r_\sigma|u-s_i(u)|_j\]
with $|u-s_i(u)|_j\neq0$ since $b_j'\neq b_j$. By the definition of $\delta$ we have  $b_j'-b_j\geq \delta\alpha$ so that 
\[m(S_{\sigma'}([-1/2,1/2]^d)\cap S_{\tau'}([-1/2,1/2]^d))\geq \delta\alpha.\]
The case $c_j'<c_j$ follows similarly. Since $m(S_{\sigma'}([-1/2,1/2]^d)\cap S_{\tau'}([-1/2,1/2]^d))\geq \delta\alpha$ we have $S_{\sigma'}\circ S_{\tau'}\in \E_\delta$.

Since
\[S_\sigma^{-1}\circ S_\tau=g^{-1}\circ S_{\sigma'}^{-1}\circ S_{\tau'}\circ h\text{ with }g,h\in\{Id,S_1,\dots,S_l\}\text{ then } S_{\sigma}^{-1}\circ S_\tau\in \cG.\]
Therefore 
\[m(S_{\sigma}^{-1}\circ S_\tau([-1/2,1/2]^d)\geq\varepsilon_2,\]
and thus
\[m(S_{\sigma}([-1/2,1/2]^d)\cap S_\tau([-1/2,1/2]^d)\geq r_\sigma\varepsilon_2\geq r_{\sigma^-}r_{\min}\varepsilon_2\geq \varepsilon\alpha\]
which contradicts \eqref{eq4.3}.
\end{proof}

%\part{Local dimension of self-similar measures}
\section{Self-similar Measures}\label{SMeasures}

This section does not follow from the previous ones but is instead a practical application of how to calculate the local dimension of self-similar measures. With this objective, we look at what was done in \cite{HareHareMatthews}, where the authors developed an algorithm to calculate the local dimension of finite-type measures in $\R$. Here, we will show that the same kind of algorithm can be expanded to $\R^d$.

We present a set of technical assumptions that emulate those presented in \cite{HareHareMatthews}. We define transition matrices; these transition matrices will encode the weight of the net intervals. Using these transition matrices, one can calculate the local dimension of the self-similar measure at a point $x$.

Following what was done in \cite{HareHareMatthews}, we concentrate all the information on a directed graph. We show that this directed graph has an essential class. This allows us to conclude that we obtain the same kind of directed graph as those in \cite{HareHareMatthews}. This means that we obtain all the good properties that they show that this kind of graph has. To see all of these properties obtained from the graphs, we refer to \cite{HareHareMatthews}.

%We remark that in this section, we will work with IFS with a self-similar set $\K$. It is not hard to change this hypercube to any other, and the results follow in the same manner. 

%{\color{red}
%I'm not sure I like the fact that most of the paper uses $[-1/2, 1/2]^d$ and this 
%section uses $\K$.  I'm also not sure it is really needed either way.
%For many of these things, we can simply use $I = \mathrm{hull}(K)$ and $\partial I$ for the boundary.  It might be worth talking about this a bit.
%}

We recall that given an IFS $\cS=\{S_i\}_{i=0}^n$ and probabilities $p_0,\dots,p_n$ ($p_i>0$ and $\sum_ip_i=1$). Then there is a unique measure associated to $\cS$ and the probabilities, satisfying  
\[\mu(A)=\sum_{i=0}^n p_i\mu\circ S_i^{-1}(A).\]
  One of the properties of interest when studying self-similar measures is the local dimension. 
\begin{defi}
    Given a measure $\mu$, we define the \emph{upper local dimension} of $\mu$ at the point $x\in \supp \mu$ as
\[\udim \mu(x)=\limsup_{r\to0}\frac{\ln(\mu(B(x,r))}{\ln(r)}.\] We similarly define the \emph{lower  local dimension} at the point $x\in \supp\mu$ as
\[\ldim\mu(x)=\liminf_{r\to0}\frac{\ln(\mu(B(x,r))}{\ln(r)}.\] If the upper and lower local dimensions coincide, then we simply call it \emph{local dimension} and use the notation $\dim\mu(x)$.
\end{defi}
Note that the local dimension is a way of measuring how much mass is concentrated around a point. The larger the local dimension, the less mass around the point. The more mass around a point, the lower the local dimension. A measure with an atomic point $x_0$ will have local dimension $0$ at the point $x_0$. 

Given an IFS $\cS$, we will give a special name to the similarities that interact with the boundaries of the attractor. 
\begin{defi}
    Let $\cS=\{S_i\}_{i=1}^n$ with attractor $\K$. We call $S_j$ a boundary similarity if $S_j(\K)\cap\partial\K\neq\emptyset$.
\end{defi}
Inspired by the standard technical assumption in \cite{HareHareMatthews}, we give the next set of technical assumptions.
\begin{defi}[Technical Assumptions]
Let $\cS=\{S_i\}_{i=0}^m$ be an IFS and $\mu=\sum p_iS_i^{-1}$ be the associated self-similar measure. We say $\mu$ satisfies the \emph{standard technical assumptions} if
    \begin{itemize}
        \item $\cS$ satisfies the FNC.
        \item The attractor $K=\K$.
        \item Let $p_{min}=\min_{i} p_i$, if $S_j$ is a boundary transformation, then $p_j=p_{min}$. 
        \item All contraction ratios are the same, i.e., $\cS$ is equicontractive.
    \end{itemize}

We note that we can omit the measure and say that the IFS satisfies or meets the technical assumptions.
\end{defi}
When an IFS $\cS=\{S_i\}_{i=0}^m$ satisfies the technical assumptions; the IFS satisfies the FTC. This follows from the definitions. Let $0<r<1$ be the contraction ration of $\cS$: then instead of considering all levels of $\delta$, it is enough to consider when $\delta=r^n$. With this in mind, for ease of notation, we write $\sS_n:=\sS_{r^n}$, $\F_n:=\F_{r^n}$, and $\Delta_n=\Delta_{r^n}$. 

We define a transition matrix from $T_{n-1,n}(x):\Delta_{n-1}(x)\to\Delta_n(x)$. We index the row by the images of $K$ covering $\Delta_{n-1}$ at level $n-1$. we do not list $S_\sigma$ if $S_\sigma=S_\tau$. The columns are indexed by the images that cover $\Delta_n$ at level $n$.

Consider $\sigma$ with $|\sigma|=n-1$ and $S_\sigma(\K)\supset\Delta_{n-1}$. Then $S_\sigma$ corresponds to the row $i$. Let $j\in\{0,1,\dots,m\}$. If $S_{\sigma j}(\K)\supset\Delta_n$, then $S_{\sigma j}$ corresponds to a column, say column $k$, then we set $(i,k)=p_j$. If for all $S_{\sigma j}$ does not correspond to column $k$, we set $(i,k)=0$. The order of columns can be done in lexicographic order.  The order itself isn't important, so long as it is consistently applied at every stage of the algorithm.

We can see how these transition matrices are obtained in Example \ref{Netintervalslookweird}.

We set
\[T_{0,n}(x)=T_{0,1}(x)T_{1,2}(x)\dots T_{n-1,n}, \text{ and}\]
\[P_n(x)=\text{sum of the entries of }T_{0,n}(x).\]
The goal is to show that $T_{0,n}(x)$ and $P_n(x)$ have all the information we need to calculate the local dimension of $x$. In particular, the tails of the sequences $\{P_n(x)\}$ and $\{T_{0,n}(x)\}$ have encoded the information necessary to calculate the local dimension of $\mu$ at $x$. In Theorem \ref{thm5.4} we show $P_n(x)\approx \mu(\Delta_n(x))$. We
will show that for a point $x$ in the boundaries of $S_\sigma(\K)$ we can choose any of the net intervals that contain $x$. 
\begin{thm}\label{thm5.4}
    Let $\cS=\{S_i\}_{i=1}^n$ be an IFS that meets the technical assumptions. Then there exists a $C>0$ such that \begin{equation}
        CP_n(x)\leq \mu(\Delta(x))\leq P_n(x).
    \end{equation}
\end{thm}
\begin{proof}
    We will show by induction that \begin{equation}
        P_n(x)=\sum_{\substack{|\sigma|=n\\ \Delta_n(x)\subset S_\sigma(\K)}}p_\sigma
    \end{equation}
    From this the result will follow. 
For $n=1$ we have
\[P_1(x)=\text{sum of entries of }T_{0,1}(x)\]
and
\[T_{0,1}=[p_i,\dots p_k]\text{ with }\Delta_1\subset S_j(\K).\]
This gives $|\sigma|=1$ with $\Delta_1\subset S_\sigma(\K)$ contributing to the sum of $P_1(x)$.

Consider $P_n(x)$ as the sum of the entries in \[T_{0,n}(x)=\underbrace{T_{0,1}(x)T_{1,2}(x)\dots T_{n-2,n-1}(x)}_{[a_1,\dots ,a_\ell]} T_{n-1,n}.\]
Here $a_1$ corresponds to the sum of al $p_\sigma$ where $S_\sigma$ corresponds to the first row of $T_{n-1,n}$. Similarly for $a_2,\dots,a_\ell$.

For $S_\tau$ with $|\tau|=n$ and $\Delta_n(x)\subset S_\tau(\K)$ we have $\tau=\sigma j$ for some $|\sigma|=n-1$. Furthermore, $S_\sigma$ matches up to one of the $a_1,\dots,a_\ell$  and so $S_\tau$ corresponds to some $b_1,\dots, b_k$ where
\[T_{0,n}(x)=T_{0,1}(x)T_{1,2}(x)\dots T_{n-1,n}=[b_1,\dots,b_k].\]
That is, if $\sigma$ matches to $a_i$, $\tau=\sigma j$ matches to $b_m$, then we have $p_j$ in the $(i,m)$ entry of the matrix.
\end{proof}
Theorem \ref{thm5.5} will solve the difficulty of choosing the net interval when $x$ is in multiple net intervals. 
\begin{thm}\label{thm5.5}
    Let $\cS=\{S_i\}_{i=1}^n$ satisfy the technical assumptions. Then there exists a $c>0$ such that if $\Delta_n(x)$ and $\Delta(y)$ are adjacent, then
    \begin{equation}\label{adjacentnet}
        \frac{1}{c n}\mu(\Delta(y))\leq \mu(\Delta(x))\leq cn\mu(\Delta(y)).
    \end{equation}
\end{thm}
    
    \begin{proof}
        From previous results it is enough to prove that there exists a $c>0$ such that
        \[P_n(x)<cnP(y).\]
        This is clearly true for $n=1$ as there are only a finite numbers of $\Delta_1$ with a constant $c_0$. Define
        \[c=\max\{c_0,1/p_{min}\}.\]
        Assume Equation \eqref{adjacentnet} is true up to $n-1$. There are two cases at level $n$, $\Delta_n(x)$ and $\Delta_n(y)$ can have the same parent or different parents.

        \begin{enumerate}[label=Case \arabic*:]
            \item Assume $\Delta_n(x)$ and $\Delta_n(y)$ have the same parent, i.e. $\Delta_{n-1}(x)=\Delta_{n-1}(y)$.
            From the construction of $c_0$ we have
            \[\sum_{\Delta_n(x)\subset S_{\sigma j}(\K)}p_{\sigma j}\leq c_0 \sum_{\Delta_n(y)\subset S_{\sigma k}(\K)}p_{\sigma k}.\]
            Adding this over all $|\sigma|=n-1$ with $\Delta_{n-1}\subset S_\sigma(\K)$ we get
            \[P_n(x)\leq c_0 P_n(y)\leq c P_n(y)\leq c n P_n(y).\]
            \item Assume $\Delta_n(x)$ and $\Delta_n(y)$ have a different parent. Then $\Delta_{n-1}(x)$ and $\Delta_{n-1}(y)$ are adjacent and $P_{n-1}(x)\leq C(n-1)P_{n-1}(y)$. We recall
            \[        P_n(x)=\sum_{\substack{|\sigma|=n-1\\ \Delta_{n-1}(x)\subset S_\sigma(\K)}}p_\sigma.\]
            We consider
            \begin{align*}
                D_1&=\{\sigma: |\sigma|=n-1, \Delta_{n-1}(x)\subset S_\sigma(\K),\Delta_{n-1}(y)\not\subset S_\sigma(\K)\},\\
                D_2&=\{\sigma: |\sigma|=n-1, \Delta_{n-1}(x)\not\subset S_\sigma(\K),\Delta_{n-1}(y)\subset S_\sigma(\K)\},\\
                E&=\{\sigma: |\sigma|=n-1, \Delta_{n-1}(x)\subset S_\sigma(\K),\Delta_{n-1}(y)\subset S_\sigma(\K)\}.\\
            \end{align*}
        \end{enumerate}
        Recall that $\sum_ip_i=1$, hence
        \begin{align*}
            P_n(x)&\leq \sum_{\sigma\in D_1}p_\sigma p_{min}+\sum_{\sigma\in E}p_\sigma\overbrace{\sum_{i=1}p_i}^{=1}\\
            &\leq p_{min}\sum_{\sigma\in D_1}p_\sigma+\sum_{\sigma\in E}p_\sigma+\overbrace{p_{min}\sum_{\sigma\in E}p_\sigma+\sum_{\sigma\in D_2}p_\sigma}^\text{add this}\\
            &=p_{min}\sum_{\sigma\in D_1\cup E}p_\sigma+\sum_{\sigma\in D_2\cup E}p_\sigma\\
            &=p_{min}P_{n-1}(x)+P_{n-1}(y)\\
            &\leq \overbrace{p_{min}c(n-1)P_{n-1}(y)}^\text{induction}+\overbrace{cp_{min}}^{>1}P_{n-1}(y)\\
            &=p_{min}n P_{n-1}(y)\\
            &\leq cnP_n(y).
        \end{align*}
        Where the last inequality is true since $p_{min}\leq p_i$ for all $i$ and hence $p_{min}P_{n-1}(y)\leq P_n(y)$. This gives
        \[P_n(x)\leq P_n(y).\]
    \end{proof}
    Theorem \ref{thm5.6} combines Theorem \ref{thm5.4} and \ref{thm5.5} to get the local dimension at a point $x$. 
    \begin{thm}\label{thm5.6}
    Let $\cS=\{S_i\}_{i=1}^n$ satisfy the technical assumptions. Then
    \begin{equation}
        \udim\mu(x)=\limsup_{n\to\infty}\frac{\log(\mu(\Delta_n(x))}{\log(r^n)}=\limsup_{n\to\infty}\frac{\log(P_n(x))}{\log(r^n)}.
    \end{equation}
    A similar result holds for the lower local dimensions and for when the local dimension exists. 
    \end{thm}
    Following what was done in \cite{HareHareMatthews} we will organize this information on a weighted directed graph. This graphs will have as vertices the normalize net intervals. We put a directed edge from $\Delta_1$ to $\Delta_2$ if $\Delta_2$ is a child of $\Delta_1$. We give this edge weight $T_{\Delta_1,\Delta_2}$. We show an example of this in Example \ref{Netintervalslookweird}. 

    In the rest of this section, we will show that the associated graphs have a unique \emph{essential class} (see Definition \ref{essential}). This will imply that the associated graphs share the same properties as the associated graphs in \cite{HareHareMatthews}.
 
    \begin{defi}\label{essential}
        We say that a collection of net intervals $\mathscr{L}$ is a loop class (strongly connected component) if for all $A,B\in \mathscr{L}$ there is a path from $A$ to $B$ staying in $\mathscr{L}$.

        We say that a loop class is \emph{maximal} if it is not contained as a proper subset of another loop class. We say that a maximal loop class is \emph{essential} if all children of the net interval remain in the loop class.
    \end{defi}
    \begin{thm}\label{thmessentialclass}
            Let $\cS=\{S_i\}_{i=1}^n$ satisfy the technical assumptions. Then the associated graph has a unique essential class.
    \end{thm}
\begin{proof}
Since the IFS satisfies the finite neighbor condition, there are finitely many net intervals. Set $\Delta$ to be the net interval of minimal volume, if there is more than one of this size, choose the one with maximal cover. Let $S_{\sigma_1},\dots,S_{\sigma_k}$ be the cover of $\Delta$. 

We will show that every net interval $\Delta'$ will have a descendant of type $\Delta$. Hence, the net interval $\Delta$ will be in the essential class, and therefore the essential class is unique.

Recall
\[\Delta=\left((S_{\sigma_1}(\K)\bigcap_{i=2}^kS_{\sigma_i}(\K)\right)\setminus \left(\bigcup \inte(S_{\sigma_l}(\K))\right).\]
Let $\Delta'$ be a net interval and $S_\tau$ such that $\Delta'\subset S_\tau(\K)$. Note that there is $\sigma$ such that $S_{\tau\sigma}([0,1])^d\subset\Delta'$, consider
\[B=\left(S_{\tau\sigma\sigma_1}(\K)\bigcap_{i=2}^kS_{\tau\sigma\sigma_i}(\K)\right)\setminus \left(\bigcup \inte(S_{\tau\sigma\sigma_l}(\K))\right).\]
Note that $B$ needs to be non-empty since $\Delta$ is non-empty ($S_{\tau\sigma}^{-1}(B)=\Delta$). Note that $B$ has normalize measure at most the same as $\Delta$ and maximal cover. Hence $B=\Delta$.
\end{proof}
\begin{remark}\label{Eclaasremark}
    It is important to note that the existence of a unique essential class is a property of the IFS and not the self-similar measure. In other words, we do not require the standard technical assumptions for the existence of a unique essential class.
\end{remark}
We note that a lot of the properties one can deduce for these types of measures are properties that are described through the graph. The graph we obtain in this setting shares the same properties as the one obtained in \cite{HareHareMatthews}. This means that we can deduce all the same properties that were deduced in \cite{HareHareMatthews}. We refer the reader to that paper to learn about the density of cyclic points and some other important properties.

Finally, we present an example.

\input{Example}
\section{Open Questions}

Throughout this paper, we have seen that the techniques and ideas that work for IFS in $\R$ can be extended to $\R^d$. But there is a cost, the geometry becomes increasingly foreign and many of the tricks that we could use in $\R$ are simply lost. In $\R^d$ rotations become a problem, whereas in $\R$ we can think of only having two rotations, the identity or the reflection of the interval. In $\R^d$ we have infinitely many rotations.

In \cite{hare/hare/rutar}, the underlying order is used to define the classical net intervals. In $\R^d$ we don't have an order. This underlying structure is used again to define the classical neighbor set and to see the equivalence between GFTC\textsubscript{co} and FNC in $\R$. In $\R^d$, we are unable to obtain this equivalence. We leave as a conjecture that these two overlapping conditions are equivalent in general. Part of the problems we saw when trying to prove this conjecture is on the rotations and how to give a normalization that incorporates more information of the IFS. 

We introduce early on the notions of net intervals and neighbor sets; we note that these are not the only ways to define them. It is worth mentioning that choosing how to normalize was done with some of the results that came later on in mind. A different choice could be made, for example choosing to normalize with respect to the volume of the net intervals. This makes it hard to know what can actually fit inside the net interval, as we have seen that they have no simple shapes. 

If we allow for rotations that do not fix the hypercube, it is unclear whether the WSC and the GFTC\textsubscript{co} are equivalent.

When working with self-similar measures, anything discussing how to weaken the technical assumptions is open in $\R^d$. Some ideas of how to proceed in a different direction can be found in \cite{Hare_Hare_ShingNg_2018} and in \cite{HARE20181653}. 

Finally, it would be interesting to see if there is anything that can be said about IFS with irrational rotations. We have only been able to work with cases where these are excluded. We believe that, in general, irrational rotations with overlaps probably will not allow for WSC.
\nocite{hare2016local}
\printbibliography[heading=bibintoc,title={References}]
\end{document}